\documentclass{article}
\usepackage{amsmath,amssymb,amsthm}
\usepackage[svgnames, table]{xcolor} 

\usepackage[ruled,vlined]{algorithm2e}
\usepackage{setspace}
\usepackage{xr-hyper}
\usepackage[colorlinks=true, linkcolor=blue, citecolor=blue, filecolor=blue, runcolor=blue, urlcolor = blue]{hyperref}

\usepackage{fullpage}
\newtheorem{definition}{Definition}[section]
\newtheorem{theorem}{Theorem}[section]
\newtheorem{proposition}{Proposition}[section]
\newtheorem{lemma}{Lemma}[section]
\newtheorem{corollary}{Corollary}[section]
\newtheorem{remark}{Remark}[section]
\newtheorem{example}{Example}[section]

\usepackage{graphicx,calc}
\DeclareGraphicsExtensions{.png}

\begin{document}

\title{A parallel algorithm for the computation of the Jones polynomial}

\author{Kasturi Barkataki$^1$ \and Eleni Panagiotou$^1$}

\date{$^1$School of Mathematical and Statistical Sciences, Arizona State University,  Tempe, 85281, Arizona, USA \today}

\maketitle

\begin{abstract}
Knots, links and entangled filaments appear in many physical systems of interest in biology and engineering. Classifying knots and measuring entanglement is of interest both for advancing knot theory, as well as for analyzing large data that become available through experiments or Artificial Intelligence. In this context, the efficient computation of topological invariants and other metrics of entanglement becomes an urgent issue. The computation of common measures of topological complexity, such as the Jones polynomial, is \#P-hard and of exponential time on the number of crossings in a knot(oid) (link(oid)) diagram. In this paper, we introduce the first parallel algorithm for the exact computation of the Jones polynomial for (collections of) both open and closed simple curves in 3-space. This algorithm enables the reduction of the computational time by an exponential factor depending on the number of processors. We demonstrate the advantage of this algorithm by applying it to knots, as well as to systems of linear polymers in a melt obtained from molecular dynamics simulations. The method is general and could be applied to other invariants and measures of complexity.
\vspace{2pc}

\noindent{\it Keywords}: {\small Parallel Algorithm, Jones Polynomial, Knots, Linkoids, Divide and Conquer} 
\end{abstract}

\maketitle

\section{Introduction}
A significant challenge in knot theory and its applications is the computation of topological invariants, which are measures that can rigorously classify them. Quantifying complexity of systems of entangled, knotted or linked, filaments is of interest also in biology and engineering.  As modeling and experiments lead to the accumulation of large data, concerning for example, proteins, understanding the role of entanglement in these systems requires extensive capabilities of computation of knot and link invariants. In this paper, we introduce the first parallel algorithm for computing the Jones polynomial of knots/knotoids/linkoids and open curves in 3-space.

The Jones polynomial, an important knot invariant (\cite{Jones1985,Kauffman1987}), is known to be \#P-hard and it is of exponential time on the number of crossings of a knot diagram (\cite{jaeger1990computational}).  For special types of knot diagrams, such as torus knots, algebraic links, pretzel links, 2-bridge links, link diagrams with bounded treewidth, 3-braid links and Montesinos links, there are various efficient classical algorithms for evaluating the Jones polynomial (\cite{diao2009jones,hara2009fast,makowsky2005coloured,jin2010computing,mighton2001computing,murakami2007fast,utsumi2002computation}). Even though important for our understanding of knots, these remain special cases of knots and, in practice, such cases may have very low probability of occurrence. Knots and entangled filamentous matter, ranging from textiles to (bio)polymers, appear in our everyday life and their topological complexity is intricately related to their mechanics and function (\cite{Arsuaga2005,Edwards2000,Liu2018,Panagiotou2019,Qin2011,Sulkowska2012, sleiman2024geometric, bonato2022topological}). 
Physical systems of interest, such as models of polymer melts, DNA or collections of biopolymers in nature, correspond to link(oid) diagrams of the order of at least $\approx 10K$ crossings. Similarly, random knots are highly unlikely to result in special type of diagrams that can be treated with the above methods (\cite{cantarella2016knot, cantarella2018open, van2011universality, chan2007mean,Diao1995,Hanse1992,Micheletti2011}.

In \cite{burton-homflypt-fpt}, an algorithm for computing the HOMFLY-PT polynomial of arbitrary links  was introduced using a fixed-parameter tractable approach, based on the treewidth of the graph associated with a given link diagram, and this was implemented in practice in the topological software package \texttt{Regina} (\cite{regina}). Specifically, the algorithm achieves a sub-exponential runtime of $e^{\mathcal{O}(\sqrt{n} \log n)}$ for an arbitrary link diagram with $n$ crossings. This result uses the fact that the graph associated with an $n$-crossing link diagram is planar, with a treewidth bound of at most $\mathcal{O}(\sqrt{n})$. Algorithms based on the number of strands, $c$, of braid representations of arbitrary links were proposed and implemented in \cite{morton1990calculating} through connections to Hecke algebras and quantum groups. The cost of such strand based algorithms scales as $\mathcal{O}(2^c)$ in memory and $\mathcal{O}(4^c)$ in time. Quantum algorithms for approximating the Jones polynomial have been introduced at specific values (roots of unity) that are polynomial in time (\cite{Aharonov_Jones_Landau2006,Aharonov_2011}). Recent advances in quantum computing have led to the successful implementation of the algorithm in \cite{Aharonov_Jones_Landau2006}, for  approximating the Jones polynomial of knot diagrams of up to 104 crossings (\cite{laakkonen2025less, castelvecchimind}). 

An approach that has not been thoroughly explored is that of parallel algorithms for the exact computation of the Jones polynomial. An underlying difficulty in reducing the computational time of the Jones polynomial for a general diagram, which also applies to other invariants, is that the skein relation associated with the bracket polynomial expansion acts recursively at its crossings, giving an exponential cost depending on the number of crossings in a diagram.  Ideally,  the computational time of the Jones polynomial could be reduced by subdivision of a diagram to allow for a divide and conquer method, to apply state sum expansion  computations independently on each subdivided piece, as it was proposed in \cite{barnatan2007fast} for faster serial computation of the Khovanov homology of a knot.  However, the last step then would require reassembly, which can be as expensive as the computation of the invariant for the entire knot diagram itself. In the algorithm for the Khovanov homology in \cite{barnatan2007fast}, a method to reduce the computation time of the reassembly step was discussed by simplifying the output of each partial computation, based on decomposing Khovanov complexes into smaller subcomplexes  to which delooping and Gaussian elimination tools were applied. In \cite{ellenberg2014efficient}, bounds on the cutwidth of a planar graph are used to to prove that, given a cutwidth realization of a graph, the Kauffman bracket of a link with $n$ crossings can be computed in time, $\mathcal{O}(\operatorname{poly}(n)2^{C\sqrt{n}})$ where, $C=6\sqrt{2}+5\sqrt{3}$. However, finding an optimal cutwidth layout for a general graph is NP-hard, though, polynomial-time approximation algorithms exist, with the best known achieving an approximation ratio of $\mathcal{O}((\log n)^{3/2})$.

In this paper, we introduce a parallel algorithm for the computation of the Jones polynomial and a parallel implementation of it which employs the divide and conquer approach and exploits a novel, simple expression of reassembly based on the formalism introduced in \cite{Barkataki2022,Barkataki2024_virtual} for linkoids. Moreover, the algorithm enables the exact computation of the Jones polynomial not only for knots, but also for knot(oid)s/link(oid)s and (collections of) open curves in 3-space which may be equally, or more important for applications to entangled physical systems. 
More precisely, in this paper, a parallel algorithm for the computation of the Jones polynomial is introduced by subdividing a diagram a knot(oid)/link(oid) diagram into smaller pieces. The smaller pieces are linkoids (open arc diagrams that are not necessarily tangles), which can be classified separately (\cite{Turaev2012,Barkataki2022, Barkataki2023pbc, Barkataki2024_virtual}) and  the resulting states from the state-sum expansion of such linkoids are grouped based on distinct permutations of the endpoints. By decomposing a knot into its constituent linkoid pieces and applying state sum expansions, computations can be performed independently on each subdivided piece, thus enabling the implementation of parallel algorithms. By deriving a closed form expression of the Jones polynomial of a link in terms of the states of its constituent linkoid subdivisions, the assembly of the results is done via a Cartesian product computation which is also implemented in a parallel algorithm. This leads to a computational time of the Jones polynomial given as $
\displaystyle T(n) = \mathcal{O}\left( 2^{\frac{n}{p}+1}\right) + \mathcal{O}\left(2^{\frac{n}{\sqrt{p}}-\log p}\right)$, where $n$ denotes the number of crossings in a knot/link diagram and $p$ denotes the number of processors. In practice, by employing subdivision that employs the cutwidth algorithm and grouping of state, the algorithm can be approximated by time complexity of the order $\mathcal{O}(2^{\frac{n}{p}})$ for large $n$ (resp. $\mathcal{O}(2^{\sqrt{32n}-\log p})$ when employing a treewidth based algorithm). 

We provide numerical results which illustrate the advantages in performance of the subdivision and gluing techniques with respect to time complexity for mathematical knots, as well as for random systems of linear polymers in a melt through molecular  dynamics simulations. A \texttt{python} package that implements the methods described in this paper, and enables the computation of the Jones polynomial for collections of open or closed curves in 3-space, is made available on \texttt{GitHub} (\cite{ParallelJonesPolynomial2025}), along with an explanatory manual text. The method is general and could be applied to parallelize the computation of other invariants, such as the Arrow polynomial and Khovanov homology (\cite{khovanov2000categorification}, \cite{dye2009virtual}).  We also discuss how this framework advances our understanding of aspects of knots in terms of properties of the constituent linkoids and for some very special cases, we remark that the Jones polynomial of a link diagram can be completely determined by the Jones polynomial of a subdivided linkoid piece. More interestingly, we show that the Jones polynomial of a knot(oid)/link(oid) can be expressed as a linear combination of the Jones polynomial over the virtual spectrum of any of its constituent linkoid parts.

The paper is organized as follows: Section \ref{sec_jones_poly} provides the necessary notations and definitions related to the Jones polynomial of links, linkoids and collections of open curves in 3-space. Section \ref{sec_par-alg} provides the parallel algorithm and its associated time complexity for the  computation of the Jones polynomial, by means of subdivision of a link(oid) diagram and grouping of states into distinct classes. Section \ref{sec-jones-pieces} addresses the topological complexity, in terms of the Jones polynomial, of link(oid) diagrams in relation to the complexity of their constituent parts. Section \ref{sec_discussion} presents the conclusions of this study.

\section{The Jones polynomial}
\label{sec_jones_poly}

Knots and links are classified into equivalence classes up to ambient isotopy. The Jones polynomial was introduced in \cite{Jones1985} to classify knots and links from their diagrams. The precise definitions follow below: 

\begin{definition}(\textit{Knot/Link and Knot/Link diagram} \cite{adams2004knot}) A \textit{link} of n components is defined to be n simple closed curves in $\mathbb{R}^3$. A link of one component is a \textit{knot}. A regular projection of a link is called a \textit{link diagram} if the overcrossing/undercrossing information is marked at every double point, which are called \textit{crossings}, in the projection. Knot/link diagrams are classified using the Reidemeister moves, which correspond to ambient isotopic knots/links in 3-space. 
\end{definition}

\begin{definition}(\textit{Jones polynomial of knot/link} \cite{Kauffman1987})
The Jones polynomial of an oriented link, $\mathcal{L}$ is defined as, \begin{equation}f_{\mathcal{L}} = (-A^3) ^{-\operatorname{Wr}(L)} \langle L  \rangle \hspace{0.05cm},\label{eq-jones-poly-link}\end{equation}
where $L$ is any diagram of $\mathcal{L}$, $\operatorname{Wr}(L)$ is the writhe of $L$ and $\langle L  \rangle$ is the bracket polynomial of $L$  which is characterized by the following skein relation and initial conditions:
\begin{equation}
\begin{split}
\left\langle\raisebox{-5pt}{\includegraphics[width=.05\linewidth]{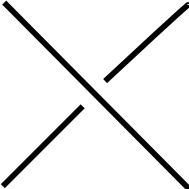}}\right\rangle=A\left\langle\raisebox{-5pt}{\includegraphics[width=.05\linewidth]{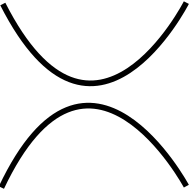}}\right\rangle +A^{-1}\left\langle\raisebox{-5pt}{\includegraphics[width=.05\linewidth]{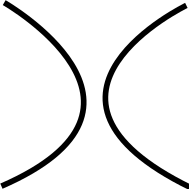}}\right\rangle,&  \bigl\langle \tilde{L}\cup \bigcirc\bigr\rangle=d\bigl\langle \tilde{L}\bigr\rangle, \langle\bigcirc\bigr\rangle = 1
\end{split}
\end{equation}
\noindent where $d=-A^2-A^{-2}$ and $\tilde{L}$ is any knot/link diagram. The \textit{ bracket polynomial} of $L$ can be expressed as the following state sum expression :
\begin{equation}
    \bigl\langle L  \bigr\rangle \hspace{0.05cm} := \hspace{0.05cm} \sum_S A^{\alpha (S)} d^{|S| -1} \quad , 
\end{equation}  where $S$ is a state corresponding to a choice of smoothing over all double points in $L$;\hspace{0.05cm} $\alpha(S)$ is the algebraic sum of the smoothing labels of $S$ and $|S|$ is the number of circular components in $S$. The Jones polynomial is obtained by substituting $A=t^{-1/4}$ in Equation \ref{eq-jones-poly-link}.
\end{definition}

Linkoids extend the diagrammatic study of links to  open arc diagrams which are classified up to diagrammatic isotopy that is described below (\cite{Barkataki2024_virtual,Barkataki2022,Gugumcu2017,Turaev2012,Gugumcu2017b,Gugumcu2021parity, Manouras2021}):

\begin{definition}
(\textit{Knotoid/linkoid diagram and Knotoid/linkoid} \cite{Barkataki2024_virtual})
 A \textit{knotoid/linkoid diagram} $L$ with $n \in \mathbb{N}$ components in $\Sigma$ is a generic immersion of $\bigsqcup_{i=1}^{n}[0,1]$ into the interior of $\Sigma$ (i.e., $S^2 = \mathbb{R}^2 \cup \infty$), with only transversal double points as singularities, each endowed with over/under crossing data. The images of $0$ and $1$ (endpoints) under this immersion are the foot and head of the component, respectively, and are distinct from each other and from the double points. A \textit{knotoid/linkoid} is an equivalence class of such diagrams under Reidemeister moves and isotopy, with the restriction that strands adjacent to endpoints cannot pass over or under a transversal strand (See Figure \ref{linkoid-moves}). \label{linkoid} A labelled linkoid can be associated with a strand permutation and a spectrum of closure permutations relating it to a spectrum of virtual links (\cite{Barkataki2024_virtual}). The \textit{strand permutation} of a linkoid $L$ with $n$ components is a permutation, $\tau$, on the set of $2n$ labelled endpoints, pairing each endpoint, $i$, with its corresponding partner, $\tau(i)$, on the same component. A \textit{closure permutation} of $L$ is a permutation, $\sigma\in S_{2n}$, on the set of endpoints that satisfies:  $\sigma^2 = \text{id}$ and $\sigma(i) \neq i$ for all $i$ (See Figure \ref{split_open_tref}). The virtual spectrum of $L$ is the set of all possible virtual closures of $L$, formed by pairing the endpoints using any closure permutation $\sigma$. Let $L_\tau$ be a linkoid with $n$ components, with strand permutation $\tau$ and a closure permutation $\sigma$ and let $E= \{ 1, 2, \cdots, 2n\}$ denote the set of the labelled endpoints of $L_\tau$. A \textit{segment cycle} of $L_\tau$, with respect to the closure permutation $\sigma$, is defined to be an orbit of a point in $E$, under the action of the group $\langle \tau, \sigma \rangle$ on $E$. The set of all segment cycles of the linkoid under the action of $\langle \tau, \sigma \rangle$ is written as the  quotient of the action, i.e, $\displaystyle E/\langle \tau, \sigma \rangle$.
\end{definition}

\begin{figure}[h!]
    \centering
    \includegraphics[scale=0.5]{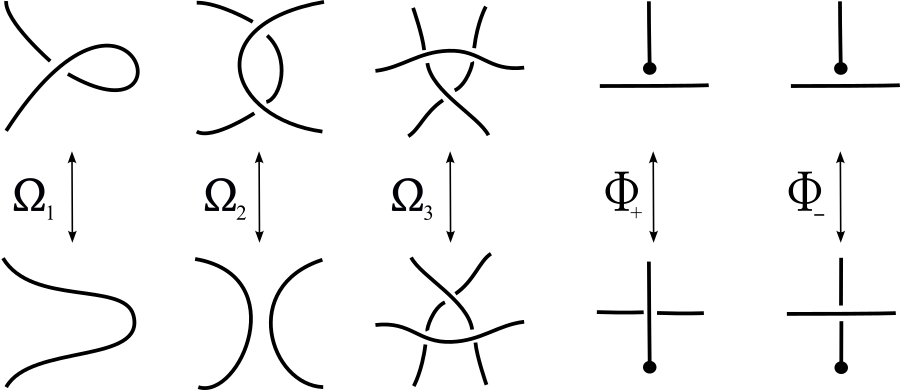}
    \caption{Omega moves (Reidemeister moves $\Omega_1,\Omega_2,\Omega_3$) and forbidden moves ($\Phi_+, \Phi_-$) on linkoid diagrams.}
    \label{linkoid-moves}
\end{figure}

The Jones polynomial of a linkoid with respect to a given closure permutation is defined as follows :

\begin{definition}(\textit{Jones polynomial of a linkoid with respect to $\sigma$} (\cite{Barkataki2022}, \cite{Barkataki2024_virtual}))
The Jones polynomial of an oriented linkoid diagram, $L$, with respect to a closure permutation, $\sigma$, is defined as, \begin{equation}f_{L^\sigma} = (-A^3) ^{-\operatorname{Wr}(L)} \langle L^\sigma \rangle \hspace{0.05cm},\end{equation}
where  $\operatorname{Wr}(L)$ is the writhe of the linkoid diagram and $\langle L^\sigma \rangle$ is the bracket polynomial of $L$, with respect to $\sigma$ which is characterized by the following skein relations and initial conditions:
\begin{equation}
\begin{split}
\left\langle\raisebox{-5pt}{\includegraphics[width=.05\linewidth]{fig/cross10.png}}\right\rangle=A\left\langle\raisebox{-5pt}{\includegraphics[width=.05\linewidth]{fig/cross30.png}}\right\rangle +A^{-1}\left\langle\raisebox{-5pt}{\includegraphics[width=.05\linewidth]{fig/cross20.png}}\right\rangle,& \hspace{0.5cm}\bigl\langle \tilde{L}\cup \bigcirc\bigr\rangle=d\bigl\langle \tilde{L}\bigr\rangle, \hspace{0.5cm} \langle\bigcirc\bigr\rangle = 1,\quad  \left\langle S_\ell \right\rangle=d^{|E/\langle \tau_S, \sigma \rangle|}\quad 
\label{bkt_sk_L}
\end{split}
\end{equation}
\noindent where $d=-A^2-A^{-2}$ ; $\tilde{L}$ is any linkoid diagram ; $S_\ell$ is the trivial linkoid formed by the collection of $n$ labelled open segments in a state, $S$, in the state sum expansion of $L$ and $\tau_S$ is the strand permutation of $S_\ell$. The term,
$|E/\langle \tau_S, \sigma \rangle|$, is the number of distinct segment cycles in $S_\ell$,  with respect to $\sigma$ and it can be alternatively expressed as :
\begin{equation}\begin{split}
   \displaystyle |E/\langle \tau_S, \sigma \rangle| &= \frac{1}{2 |\sigma \tau_S|}\left[ 2n + \sum_{p=1}^{|\sigma \tau_S|-1}  \sum_{\{\mathcal{C}_i : p \equiv 0 \operatorname{mod} l(\mathcal{C}_i)\} } l(\mathcal{C}_i) \right ].\\
   \end{split}\label{eq-C-count2}
\end{equation}
where $|\sigma \tau_S|$ is the order of $\sigma \tau_S$, $\prod_i \mathcal{C}_i$ denotes the cycle decomposition of $\sigma \tau_S$ and $l(\mathcal{C}_i)$ denotes the length of $\mathcal{C}_i$ ( See proof of Proposition \ref{count-prop} in Appendix).

The \textit{generalized bracket polynomial} of $L$, with respect to $\sigma$, can be expressed as the following state sum expression :
\begin{equation}
    \bigl\langle L^{\sigma} \bigr\rangle \hspace{0.05cm} := \hspace{0.05cm} \sum_S A^{\alpha (S)} d^{|S| + |E/\langle \tau_S, \sigma \rangle| -1} \quad , 
\end{equation}
\noindent where $S$ is a state corresponding to a choice of smoothing over all double points in $L$;\hspace{0.05cm} $\alpha(S)$ is the algebraic sum of the smoothing labels of $S$ and $|S|$ is the number of circular components in $S$ (See Example \ref{ex-linkoid-jones}). 
\label{def_jones}
\end{definition}

\begin{example}
    The Jones polynomial of the linkoid diagram $L_1$ in Figure \ref{split_open_tref}, with respect to closure permutation, $(1 \quad 2)(3 \quad 4)$, is equal to $-A^{-6}-A^{-8}+A^{-12}-A^{-16}$. Similarly, the Jones polynomial of the linkoid diagram $L_2$  with respect to closure permutation, $(5 \quad 6)(7 \quad 8)$, is equal to $-A^{-2}-A^{-10}$. 
    \label{ex-linkoid-jones}
\end{example}

Open curves in 3-space can be associated to a spectrum of linkoids (\cite{Panagiotou2020b,Barkataki2022}). The Jones polynomial of a collection of open curves is defined as follows :
\begin{definition}(Jones polynomial of a collection of open curves in 3-space (\cite{Barkataki2022}))
    Let $\mathcal{L}$ denote a collection of $m \in \mathbb{N}$ open curves in 3-space. Let $L_\xi$ denote the projection of $\mathcal{L}$ on a plane with normal vector $\xi$. The normalized bracket polynomial of $\mathcal{L}$ is defined as
\[
f_L = \frac{1}{4\pi} \int_{\xi \in S^2} (-A^3)^{-Wr(L_\xi)} \langle L_\xi \rangle \, dS, \tag{4.1}
\]
where each \( L_\xi \) is a linkoid diagram and its bracket polynomial can be calculated using Definition \ref{def_jones}. Note that the integral is taken over all vectors \( \xi \in S^2 \) except a set of measure zero (corresponding to the irregular projections). This gives the Jones polynomial of a collection of open curves in 3-space with the substitution \( A = t^{-1/4} \).\label{def-jones-open}
\end{definition}

The coefficients of the Jones polynomial of open curves in 3-space are continuous functions of the curve coordinates that tend to the classical topological invariant when the endpoints of the chains tend to coincide to form a closed link (\cite{Panagiotou2020b,Barkataki2022}). 

\section{A parallel algorithm for the Jones polynomial}
\label{sec_par-alg}

This section introduces a method to express the Jones polynomial of a link(oid) diagram in terms of the bracket polynomial states of its constituent linkoids. This is achieved by subdividing the original diagram into a collection of disjoint linkoid diagrams with fewer crossings. Theorem \ref{prop-bkt-split} gives a closed form expression  for the Jones polynomial of a link(oid) diagram, $L$, which is subdivided into $2^m$ pieces where $m \in \mathbb{N}$ (See Figure \ref{split_open_tref}). This leads to a parallel algorithm on $p=2^m$ processors (Algorithm \ref{alg:parallel_jones}) for the computation of the Jones polynomial by employing optimal subdivision and combining (grouping) of states of linkoids (Theorem \ref{prop-par-avg}).

\begin{figure}[h!]
    \centering
    {\includegraphics[scale=0.45]{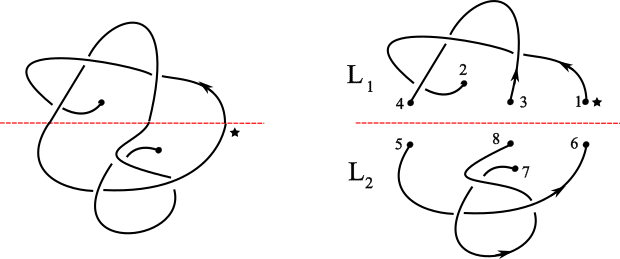}}
    \caption{A diagram of a link(oid) can be seen as the gluing of two linkoid diagrams.  (Left) A diagram of a pure knotoid, $K$, and the axis along which it is to be split with the starting point of $L_1$ denoted $\star$. (Right) The two subdivided linkoid pieces, $L_1$ and $L_2$, with strand permutations $(1 \quad 2)(3 \quad 4)$ and  $(5 \quad 6)(7 \quad 8)$, respectively, on the labelled endpoints obtained upon the splitting. The closure permutation $(4 \quad 5)(8 \quad 3)(6 \quad 1)$ on $L_1 \cup L_2$ retrieves $K$. The Jones polynomial computation for each of the two components is done in parallel and the results are combined, also in parallel, to give the Jones polynomial of the knot.}
    \label{split_open_tref}
\end{figure}

\begin{definition}(Subdivision of a link(oid))  
Let $L$ be a linkoid diagram with $n$ crossings. A subdivision of $L$ is a diagram $L^{'}$, resulting from $L$ by removing a finite number of points from the arcs of $L$, resulting in a disjoint union of $k \geq 2$ linkoid diagrams $L_i$ (for $i=1, 2, \dots, k$). Each $L_i$ is called a subdivided piece of $L$.
\label{def-subdiv}
\end{definition}

The process of combining the constituent linkoids $L_i$ to give a knot/link $L$ can be made precise via a gluing operation.

\begin{definition}(Gluing Permutation) Let $L_1$ and $L_2$ be linkoid diagrams with $n_1$ and $n_2$ components, respectively, and $\tau_1$, $\tau_2$ be the respective strand permutations on their labelled endpoints. A gluing permutation $\sigma$ is a product of $k \leq n_1 + n_2$ disjoint transpositions on the endpoints of $L_1$ and $L_2$, where at least one transposition $(i \ j)$ satisfies $i \in L_1$ and $j \in L_2$ (or vice versa). The permutation $\sigma$ introduces a virtual closure arc between each $i$ and $j$ in the transposition. The resultant glued linkoid diagram (which will have strictly fewer than $n_1 + n_2$ components) is denoted $L_1 \cup_\sigma L_2$. \end{definition}

The following theorem gives a closed formula for the expression of the Jones polynomial of a link(oid) as a function of the states of its constituent linkoids. 


\begin{theorem}
Let $\displaystyle L={\cup_{\sigma}}_{i=1}^{2^m} L_i$ denote a linkoid $L$ obtained by the gluing of $2^m$ disjoint linkoid diagrams, $L_i$, where $1\leq i \leq 2^m$ and  $m\in \mathbb{N}$, with gluing permutation $\sigma$.
The Jones polynomial of $L$ is $f_L=(-A^3)^{-Wr(L)}\langle L \rangle$, where the bracket polynomial of $L$ can be expressed in terms of the states of the constituent linkoids $L_i$ as follows:\begin{equation*}
    \begin{aligned}
        \left\langle L \right\rangle
        &= \sum_{\mathcal{S}=\left( S_{i_{k_i}} \right)_{i=1}^{2^m}} 
       \left( \left(\prod_{i=1}^{2^m}A^{\alpha(S_{i_{k_i}})} 
        d^{|S_{i_{k_i}}|} \right) \times d^{\left| E / \left\langle \prod_{k=1}^{2^m} \tau_{S_{i_{k_i}}}, \sigma \right\rangle \right|-1}\right)
    \end{aligned}
\end{equation*}
where ${\cal{S}}$ is a state of $L$ expressed as a gluing of states of the constituent linkoids. More precisely,   ${\cal{S}}$  is a tuple of states $S_{i_{k_i}}$, for $i=1, \dotsc,2^m$, $k_i=1,\dotsc, 2^{n/2^m}$, where $S_{i_{k_i}}$ represents a state with state permutation  $\tau_{S_{i_{k_i}}}$ in the state sum expansion of $L_i$, \hspace{0.05cm} $\alpha(S_{i_{k_i}})$ is the algebraic sum of the smoothing labels of $S_{i_{k_i}}$, \hspace{0.05cm} $|S_{i_{k_i}}|$ is the number of circular components in $S_{i_{k_i}}$, \hspace{0.05cm}  $E$ is the set of all endpoints in $\cup_{i=1}^{2^m} L_i$ and $\left| E / \left\langle \prod_{k=1}^{2^m} \tau_{S_{i_{k_i}}}, \sigma \right\rangle \right|$ is the number of distinct segment cycles in $\mathcal{S}$.
\label{prop-bkt-split}
\end{theorem}
\begin{proof}
Given that $ L={\cup_{\sigma}}_{i=1}^{2^m} L_i$, is a gluing of linkoids $L_i$, where $i=1, \dotsc, 2^m$ with respect to gluing permutation $\sigma$, any state of $L$ can be expressed as a gluing of states of its constituent linkoids also with respect to $\sigma$, i.e. ${\cup_{\sigma}}_{i=1}^{2^m} S_{i_{k_i}}$.  Since $\tau_{S_{i_{k_i}}}$ (the state permutation of the $k_i^{th}$ state of linkoid $L_i$) permutes only the endpoints of $L_i$,  the state permutation of ${\cup_\sigma}_{i=1}^{2^m} S_{i_{k_i}}$ is a product of disjoint permutations, given by $\tau_{{\cup_\sigma}_{i=1}^{2^m} S_{i_{k_i}}} = \prod_{i=1}^{2^m} \tau_{S_{i_{k_i}}}$ and the number of circular components $\left|{\cup_\sigma}_{i=1}^{2^m} S_{i_{k_i}}\right|$ is the sum of the individual circular components, i.e., $
\left|{\cup_\sigma}_{i=1}^{2^m} S_{i_{k_i}}\right| = \sum_{i=1}^{2^m}\left| S_{i_{k_i}}\right|$.
The number of distinct segment cycles in the state ${\cup_\sigma}_{i=1}^{2^m} S_{i_{k_i}}$ is given as, $\left|E/\left\langle \prod_{i=1}^{2^m} \tau_{S_{i_{k_i}}}, \sigma \right\rangle\right|$ (See Definition \ref{def_jones}). 
 Thus, the bracket polynomial of $L$ can be expressed as follows.
\begin{equation}\begin{split} 
  \langle L \rangle 
  &=\sum_{\mathcal{S}=\left( S_{i_{k_i}} \right)_{i=1}^{2^m}}  A^{\sum_{i=1}^{2^m} \alpha(S_{i_{k_i}})} 
    \left\langle {\cup_{\sigma}}_{i=1}^{2^m} S_{i_{k_i}} \right\rangle \\
     &= \sum_{\mathcal{S}=\left( S_{i_{k_i}} \right)_{i=1}^{2^m}}   \left(A^{\sum_{i=1}^{2^m} \alpha(S_{i_{k_i}})} d^{\sum_{i=1}^{2^m}|S_{i_{k_i}}|} \times d^{\left|E/\left\langle \prod_{i=1}^{2^m} \tau_{S_{i_{k_i}}}, \sigma \right\rangle\right|-1}\right).
     \end{split}\label{Kbkt_plain}
 \end{equation}
\end{proof}

A parallel algorithm for the computation of the Jones polynomial of a link(oid) diagram can be obtained that leads to the computational time in Theorem \ref{thm-prop-par-avg}.



\begin{theorem}
    Let $L$ be a link(oid) diagram with $n$ crossings. The Jones polynomial of $L$ can be computed via Algorithm \ref{alg:parallel_jones} in a system with $2^m$ processors with time complexity
\begin{equation*}
    \begin{split}
\displaystyle T(n) &=\mathcal{O}\left( 2^{\frac{n}{2^m}+1}\right) + \mathcal{O}\left(2^{\frac{n}{\sqrt{2^m}}-m}\right).
    \end{split}
    \end{equation*}
\label{thm-prop-par-avg}
\end{theorem}
\begin{proof}
Let $L$ be a link(oid) diagram. Consider (any) subdivision of $L$ into $2^m$ linkoids, $L_i$, $1\leq i \leq 2^m$, of $\frac{n}{2^m}$ (or $\lfloor n/{2^m}\rfloor$) crossings each. The total time complexity of doing $m$ levels of partitioning $n$ crossings by finding planar edge separators is linear i.e. $\mathcal{O}(mn)$.  By recursively applying the skein relation, the state expansion of each $L_i$ is evaluated with a time complexity of  $\mathcal{O}\left( 2^{\frac{n}{2^m}+1}\right)$, which is done in parallel for all linkoid pieces over $2^m$ processors. The states of each linkoid consist of circles and strands that correspond to state permutations.   Next, distinct states of linkoids with equivalent state permutations on the endpoints can be consolidated (grouped) eliminating redundancy, which can occur when states with the same state permutation but different loops appear (See Section \ref{sec-state-perms-seg-cyc} of Appendix).  By  extracting a $d$ factor for each circle and by adding the factors of states with the same strand permutations, we obtain only distinct state permutations with an associated polynomial in linear time. Let $\mathcal{S}_i^{(0)}$ represent the resolved states of $L_i$, and let $\mathcal{P}_i^{(0)}$ represent the set of distinct state permutations within these resolved states, where $1\leq i \leq 2^m$. The recombination of states is performed in parallel and recursively over $m$ iterations as follows (in analogy with gluing linkoids pairwise in $m$ iterations): at the $ j^{th}$ iteration, $\mathcal{P}^{(j)}_{u}$ is defined as the set of distinct state permutations obtained from the product $\mathcal{P}^{(j-1)}_{2u-1} \times \mathcal{P}^{(j-1)}_{2u}$, which consists of state permutation patterns from the $(j-1)^{th}$ iteration that share common endpoints, where $1 \leq j \leq m $ and $1 \leq u \leq 2^{m-j}$. After $m$ iterations, the final set $\mathcal{P}_{1}^{(m)}$ contains all the distinct state permutations of $L$.  For any $1 \leq j \leq m$ and $1 \leq u \leq 2^{m-j}$ the sets $\mathcal{P}^{(j)}_{u}$  can be computed via parallelization over $2^m$ processors with a time complexity of $ \mathcal{O}\left(\frac{|\mathcal{P}_{2u-1}^{(j-1)}|.|\mathcal{P}_{2u}^{(j-1)}|.2^{m-j}}{2^m}\right)$. The time complexity of computing the writhe of $L$ is $\mathcal{O}(n^2)$. Thus, the net time complexity of evaluating the Jones polynomial of $L$ can be expressed as follows:
\begin{equation}
    \begin{split}
       T(n)
 &= \mathcal{O}(mn)+\mathcal{O}\left( 2^{\frac{n}{2^m}+1}\right) +\sum_{j=1}^{m}\mathcal{O}\left(|\mathcal{P}_{2u-1}^{(j-1)}|.|\mathcal{P}_{2u}^{(j-1)}|.2^{-j}\right)+ \mathcal{O}(n^2). \\
    \end{split}
    \label{eq-min-grp-1}
\end{equation}


To estimate $|\mathcal{P}_u^{(j)}|$ we need to determine the number of strands involved at each level of gluing of linkoids. Without loss of generality, suppose $L$ is subdivided in the form of a $2^{m/2}\times 2^{m/2}$ grid, with $2^m$ pieces, such that the endpoints are distributed uniformly across all the four boundaries of each individual grid piece. (For simplicity, we may consider this to be a square grid, but it may not look geometrically as square.) Without loss of generality we also assume that each piece has $c$ strands. 
The $2^{m/2}$ pieces in any given row in the grid can be glued in $m/2$ stages. At the $j^{th}$ stage of gluing in a row, all the $2^{m/2-j}$ pieces are evaluated in parallel over $2^m$ processors. Any piece (connected unit) in the $j^{th}$ stage  of gluing by row has $c_j = \frac{c(1+2^j)}{2}$ strands and it can be obtained within time complexity $\mathcal{O}\left(\operatorname{C}_{c_{j-1}}^2 2^{-(\frac{m}{2}+j)}\right)$, where $\operatorname{C}_{c_{j-1}}$ is the $c_{j-1}^{th}$ Catalan number (See Corollary \ref{corr-any-ktang} in Appendix). 
Once each row of the grid becomes a long, connected piece, the final task is the gluing of these $2^{m/2}$ rows, which is again performed in $m/2$ stages. A piece in the $j^{th}$ level of gluing by column has $\hat{c}_j = \frac{c(2^{m/2}+2^j)}{2}$ strands and it can be obtained within time complexity $\mathcal{O}\left(\operatorname{C}_{\hat{c}_{j-1}}^2 2^{-(\frac{m}{2}+j)}\right)$, where $\operatorname{C}_{\hat{c}_{j-1}}$ is the $\hat{c}_{j-1}^{th}$ Catalan number. Using Stirling's formula, $\operatorname{C}_j \sim \frac{4^j}{j^{3/2}\pi}$, for the $j^{th}$ Catalan number, the time complexity $T(n)$ of Equation \ref{eq-min-grp-1} is given as follows:
\begin{equation}
\begin{split}
\displaystyle T(n) &= \mathcal{O}\left( 2^{\frac{n}{2^m}+1}\right) + \mathcal{O}\left(\sum_{j=1}^{m/2}  \left(\operatorname{C}_{c_{j-1}}^2 + \operatorname{C}_{\hat{c}_{j-1}}^2\right)2^{-(\frac{m}{2}+j)}\right)\\
&= \mathcal{O}\left( 2^{\frac{n}{2^m}+1}\right) + \mathcal{O}\left(2^{2c\sqrt{2^m}-m}\right).
\end{split}
 \label{eq-min-group-Tn-formal}
\end{equation}
We can assume that the number of of strands in a linkoid piece is at most half the number of crossings in the piece, $c\leq \frac{n}{2^{m+1}}$, i.e. that every strand is involved in at least crossing. Note that if that is not the case, the strand contributes trivially in counting distinct state permutations.  Note also that grouping of states  will always occur when $c\leq n / 2^m$ ( See Lemma \ref{lem-max-c-for-n} in Appendix). Therefore, the maximal number of distinct state permutations, $\operatorname{C}_c < 2^{\frac{n}{2^m}}$ and the net time complexity of computing the Jones polynomial of $L$, by employing grouping of states, in a system with $2^m$ processors is estimated as follows:
\begin{equation}
\begin{split}
\displaystyle T(n) &= \mathcal{O}\left( 2^{\frac{n}{2^m}+1}\right) + \mathcal{O}\left(2^{\frac{n}{\sqrt{2^m}}-m}\right).
\end{split}
 \label{eq-min-group-last}
\end{equation}
\end{proof}




\begin{algorithm}
\SetAlgoLined
\caption{Parallel Algorithm for Jones Polynomial}
\label{alg:parallel_jones}
\textbf{Input:} Linkoid diagram $L$ with $n$ crossings and $m$ levels of subdivision.\\
\textbf{Output:} $J(L)$: Jones polynomial of $L$.\\
\textbf{Step 1: Subdivide $L$ into $2^m$ pieces (using a Minimum Bisection Algorithm).}\\
Initialize array to store the subdivided pieces as $\Lambda_{current} \leftarrow [L]$.\\
\For{$i = 1$ to $m$ \textbf{in parallel}}
    {Set $\Lambda_{next} \leftarrow [\quad ]$\\
    \For{$\lambda$ in $\Lambda_{current}$}
        {Minimally bisect $\lambda$ into linkoid pieces $\lambda_1, \lambda_2$.\\
        Append $\lambda_1$ and $\lambda_2$ to $\Lambda_{next}$.}
        $\Lambda_{current} \leftarrow \Lambda_{next}$.}
\textbf{Step 2: Compute the bracket state sum expansion for each piece in parallel and group the states according to distinct permutations.}\\
Initialize array, $\mathcal{S}= [ \quad ]$, to store the bracket states of the pieces, $L_i$, in the array $\Lambda_{current}$, in terms of distinct state permutations.\\
\For{$i = 1$ to $2^m$ \textbf{in parallel}}
    {$\mathcal{S}^{(0)}_i \leftarrow \text{Bracket\_States}(L_i)$.\\
    $\mathcal{P}^{(0)}_i \leftarrow  \text{States in $\mathcal{S}^{(0)}_i$ with coefficients grouped}$ with respect to distinct state permutations.\\
    Append $\mathcal{P}^{(0)}_i$ to $\mathcal{S}$}
\textbf{Step 3: Combine the bracket states of all pieces in parallel over $m$ iterations to obtain the bracket polynomial of $L$}\\
\For{$j = 1$ to $m$ \textbf{in parallel}}
    {\For{$u = 1$ to $2^{m-j}$ \textbf{in parallel}}
       {$\mathcal{P}_u^{(j)} \leftarrow$ States in $\mathcal{P}_{2u-1}^{(j-1)} \times \mathcal{P}_{2u}^{(j-1)}$ with coefficients grouped with respect to distinct state permutations.}
    }
Return the final bracket polynomial $B_{\text{combined}} = \mathcal{P}_1^{(m)}$\\
\textbf{Step 4: Normalize the bracket polynomial using the writhe of $L$}\\
Compute $\text{Wr} = \text{Writhe}(L)$\\
$J(L) = \text{Normalize}(B_{\text{combined}}, \text{Wr})$\\
\Return{$J(L)$}
\end{algorithm}


An algorithm to compute the Jones polynomial  in an approximately improved computational time is discussed next. This algorithm is based on  optimizing the subdivision of a link(oid) diagram which involves applying a minimum bisection approximation algorithm to subdivide the diagram, followed by grouping the states - derived from the state-sum expansion within each piece - based on distinct permutations(patterns of the long segments) of the endpoints of each respective piece. More precisely, the proposed algorithm focuses on identifying optimal subdivisions while minimizing the number of strands in each piece. 



\begin{corollary}
    Let $L$ be a link(oid) diagram with $n$ crossings. The Jones polynomial of $L$ can be computed via Algorithm \ref{alg:parallel_jones} in a system with $2^m$ processors of time complexity approximated by
    
\begin{equation}
    \begin{split}
\displaystyle T(n) &\approx \mathcal{O}\left( 2^{\frac{n}{2^m}+1}\right) + \mathcal{O}\left(2^{C\sqrt{32n}-m}\right),
    \end{split}
     \label{eq-min-group-last}
    \end{equation}

\noindent where $C=6\sqrt{2}+5\sqrt{3}$.
\label{prop-par-avg}
\end{corollary}
\begin{proof}
Let $L$ be a link(oid) that is subdivided in $L_i$, $i=1,\dotsc,2^m$ linkoids, of at most $c$ strands each. The time complexity of the Jones polynomial is given by $T(n)  = \mathcal{O}\left( 2^{\frac{n}{2^m}+1}\right) + \mathcal{O}\left(2^{2c\sqrt{2^m}-m}\right)$ (See Theorem \ref{thm-prop-par-avg}). 
By associating a graph to the link(oid) diagram by treating each crossing as a vertex and the arcs connecting successive crossings as edges, we can employ cutwidth to count the number of strands in each piece. For any graph with $n$ vertices, it can be shown that it has a cutwidth of at most $(6\sqrt{2}+5\sqrt{3})\sqrt{n}$ (\cite{djidjev2003crossing, gazit1990planar}). By accounting that a link(oid) diagram with $n$ crossings is split into $2^m$ pieces, as  a $2^{m/2}\times 2^{m/2}$ square grid, it can be inferred that each $L_i$ has $C\sqrt{\frac{n}{2^{m-5}}}$ boundary points and  $C\sqrt{\frac{n}{2^{m-3}}}$ strands, where $C=6\sqrt{2}+5\sqrt{3}$.  
The required expression is obtained by substituting $c=  C\sqrt{\frac{n}{2^{m-3}}}$.
\end{proof}

\begin{remark}
 Finding appropriate subdivisions of a knot(oid) (resp. link(oid)) that minimize the number of strands in each linkoid can be solved by employing a minimum bisection algorithm for a corresponding graph.
The minimum bisection problem on graphs is an NP-hard problem which involves finding a graph bisection of minimum size, where the  bisection size is defined to be the number of edges connecting the two resulting sets. In practice, there exist polynomial-time algorithms (in the size of the input graph) that can approximate the solution (\cite{feige2002polylogarithmic}).
\texttt{Python} implementations of the minimum bisection algorithm of a graph (\cite{NetworkX,kernighan1970efficient}).
\end{remark}

The \texttt{GitHub} package, \textit{Parallel Jones Polynomial} (\cite{ParallelJonesPolynomial2025}) implements Algorithm \ref{alg:parallel_jones} (with minimum bisection) in \texttt{python}. This algorithm computes the Jones polynomial in parallel for  knots, links, knotoids, linkoids and open curves in 3-space. This algorithm realizes the approximate time complexity of Corollary \ref{prop-par-avg}.  

To demonstrate the computational effectiveness of the parallel algorithm, we apply it to knots of increasing number of crossings in which Algorithm \ref{alg:parallel_jones} is employed to perform the computation of the Jones polynomial with respect to 2 subdivisions of a given knot diagram.  The results shown in Figure \ref{fig_time_comp} demonstrate the superior performance of the parallel algorithm compared to the serial method, particularly for cases with larger crossing numbers. For small crossing numbers the  floating point operations (FLOPS) involved in subdivisions and gluing operations dominate the computational time over the gain of parallelization. The exponential nature of the difference between the fits in Figure \ref{fig_time_comp} confirms that the advantage of parallel processing increases with the number of crossings in a diagram.

\begin{figure}[h!]
\centering
\includegraphics[scale=0.45]{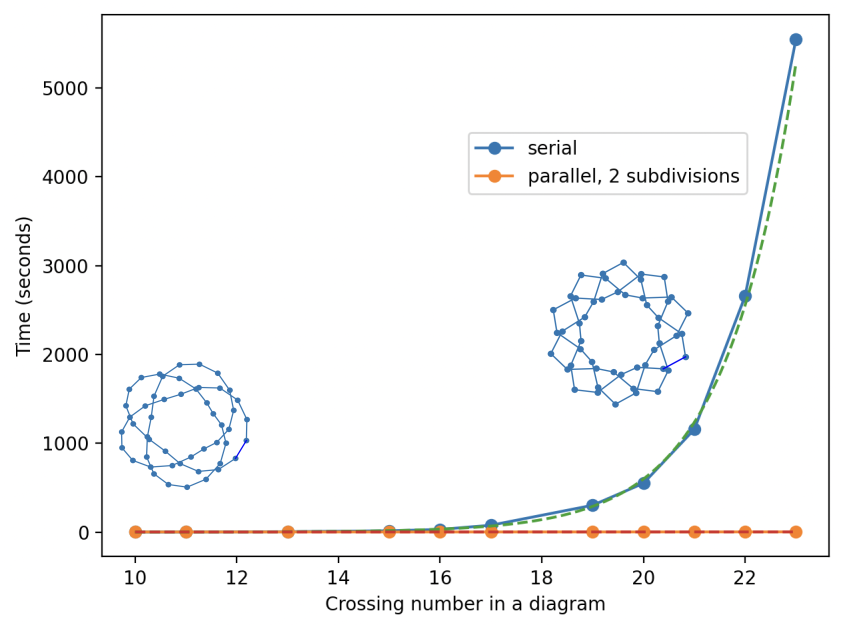}
    \caption{Computational time (in seconds) of the Jones polynomial as a function of the number of crossings in a diagram using the serial algorithm and the parallel algorithm, Algorithm \ref{alg:parallel_jones}, with 2 subdivisions and 2 processors.}
    \label{fig_time_comp}
\end{figure}

\begin{figure}[h!]
\centering
\includegraphics[scale=0.4]{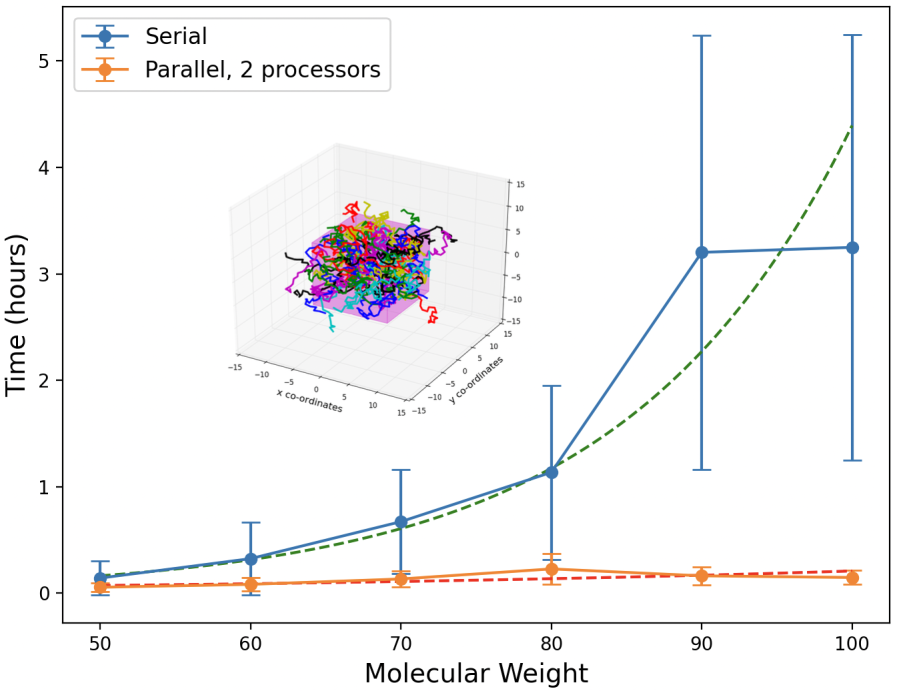}
\caption{Computational time (in hours) of the Jones polynomial of linear chains in a polymer melt as a function of molecular weight using the serial algorithm and the parallel algorithm, Algorithm \ref{alg:parallel_jones}, with 2 subdivisions and 2 processors.}
    \label{fig_poly_time_comp}
\end{figure}

Physical systems of entangled filaments, consist of linear polymers whose diagrams have many crossings, and are in a sense random and non-symmetric. Figure \ref{fig_poly_time_comp} shows a comparison between the serial and the parallel algorithms in terms of the computational time of the Jones polynomial with respect to increasing molecular weight of the constituent chains in polymer melt systems of linear polymers.  These are collections of open curves in 3-space whose Jones polynomial is evaluated using Definition \ref{def-jones-open}.  The data for the calculation is obtained by MD simulations from \cite{Panagiotou2013b} and the Jones polynomial per chain is obtained by averaging over 150 projections. We observe a significant reduction of computational time with the parallel algorithm even with only two processors.


The parallel algorithm for the Jones polynomial can be further improved, as described below in Corollary \ref{corr-sub-exp-par}, by applying the treewidth based fixed parameter tractable algorithm of \cite{burton-homflypt-fpt} (See Section \ref{sec-treewidth} of Appendix) to a link diagram at the level of its subdivided linkoid pieces. 

\newpage

\begin{corollary}\label{corr-sub-exp-par}
Let $L$ be a link(oid) diagram with $n$ crossings. The Jones polynomial of $L$ can be computed via the parallel Algorithm \ref{alg:parallel_jones} in a system with $2^m$ processors of time complexity approximated by

\begin{equation}
    \begin{split}
      \displaystyle T(n) &\approx \mathcal{O}\left(n^3 2^{\sqrt{\frac{9n}{2^m}}-m}\right) + \mathcal{O}\left( e^{\sqrt{\frac{n}{2^m}} \log \frac{n}{2^m}}\right) + \mathcal{O}\left(2^{C\sqrt{32n}-m}\right),
    \end{split}
    \label{eq-time-treewidth}
    \end{equation}

\noindent where $C=6\sqrt{2}+5\sqrt{3}$.
\end{corollary}

\begin{proof}
    Let $ L={\cup_{\sigma}}_{i=1}^{2^m} L_i$, be link(oid) diagram obtained as a gluing of linkoid diagrams $L_i, i=1, \dotsc, 2^m$, consisting of $ \approx \frac{n}{2^m}$ crossings each, with respect to gluing permutation, $\sigma$, on the total set of endpoints, $E$. A nice tree decomposition of the graph $\Gamma(L_i)$, associated with the link diagram $L_i$, has treewidth at most  $\sqrt{\frac{n}{2^m}}$ and can be constructed in time at most $\displaystyle \mathcal{O}\left(n^3 2^{\sqrt{\frac{9n}{2^m}}-m}\right)$(\cite{cygan2015michalpilipczuk, kloks1994treewidth, burton-homflypt-fpt}).   The Jones polynomial of $L$ can be expressed in terms of the leaves of the nice tree decompositions of its constituent linkoids, $L_i$, as follows:      Let $\ell_{L_i}$ be a leaf of the decision tree of $L_i$. Then $\ell_{L_i}$ is a descending linkoid with respect to $L$ (see Definition \ref{def-desc-linkoid} in Section \ref{sec-treewidth} of Appendix). Let $\ell = \cup_\sigma \ell_{L_i}$ be the diagram obtained as a gluing of $\ell_{L_i}$, $i=1,\dotsc, 2^m$ with respect to $\sigma$. Then, $\ell$ will be a descending link/linkoid diagram.  Then, in analogy to the result in Theorem \ref{prop-bkt-split}, the Jones polynomial of $L$ can be expressed in terms of states of its constituent  linkoids, $L_i$, $i=1, \dotsc, 2^m$, as follows: $\displaystyle f_L = \sum_{\ell}  \left(\prod_{i=1}^{2^m} \operatorname{g}_{\ell_{L_i}}(t)d^{|\ell_{L_i}|}\right) d^{|E/\langle \tau_\ell, \sigma \rangle|-1}$, where the sum is taken over all $\ell = \cup_\sigma \ell_{L_i}$ where $\ell_{L_i}$ is a leaf of of the decision tree of $L_i$, $\operatorname{g}_{\ell_{L_i}}(t)$ denotes the Laurent polynomial in $t$, corresponding to the resolution sequence of $\ell_{L_i}$, $|\ell_{L_i}|$ is the number of circular components in $\ell_{L_i}$, $\tau_\ell = \prod_{i=1}^{2^m}\tau_{\ell_{L_i}}$ is the product of the state permutations, $\tau_{\ell_{L_i}}$, associated with $\ell_{L_i}$ and $|E/\langle \tau_\ell, \sigma \rangle|$ denotes the number of segment cycles in $\ell$.
    
   The contribution from each $L_i$ (each term in the product inside the parenthesis) can be processed in parallel over $2^m$ processors using the fixed parameter tractable approach, contributing a time complexity of $\mathcal{O}\left( e^{\sqrt{\frac{n}{2^m}} \log \frac{n}{2^m}}\right)$. As explained in the proof of Corollary \ref{prop-par-avg}, each $L_i$ consists of at most $c=  C\sqrt{\frac{n}{2^{m-3}}}$ strands,  where $C=6\sqrt{2}+5\sqrt{3}$, implying that the recombination of contributions over all $L_i$ can be performed in parallel with a time complexity of $\mathcal{O}\left(2^{C\sqrt{32n}-m}\right)$ to get the final Jones polynomial expression of $L$.   
\end{proof}

From Equation \ref{eq-min-group-last}, it can be seen that, for a fixed $m\in \mathbb{N}$, the first term, i.e. $\mathcal{O}\left( 2^{\frac{n}{2^m}+1}\right)$, is the dominant term of the time complexity for sufficiently large $n$. When a treewidth based algorithm is employed, the dominant term of the time complexity is given by the term $\mathcal{O}\left(2^{C\sqrt{32n}-m}\right)$ of Equation \ref{eq-time-treewidth}.

\begin{remark}
    The parallel algorithm in Corollary \ref{prop-par-avg} reduces the memory usage per processor to $\mathcal{O}(2^{n/2^m})$, as each processor handles only the states and results of its assigned piece, rather than the entire diagram, in which case the memory usage is $\mathcal{O}(2^n)$. Since there are $m$ steps in the  gluing process, the memory usage per processor is $\mathcal{O}(m 2^{C\sqrt{32n}-m})$.
\end{remark}

\begin{remark}
    Note that, the techniques of subdivision and grouping of states are beneficial even when incorporated into a serial algorithm of computing the Jones polynomial of a link(oid) diagram, as it was proposed in \cite{barnatan2007fast} for faster serial computation of the Khovanov homology of a knot. It can be shown that the resultant time complexity,
$  T(n) = 2^m \mathcal{O}(2^{\frac{n}{2^m}+1})+ \mathcal{O}\left(2^{C\sqrt{32n}}\right)$, which is more efficient with respect to the time complexity, $\mathcal{O}(2^n)$, of the serial algorithm without subdivision and grouping. 
\end{remark}


\section{Jones polynomial of a knot in terms of the Jones polynomial of its constituent linkoids}
\label{sec-jones-pieces}

\noindent  The expression in Theorem \ref{prop-bkt-split}, for the Jones polynomial of a knot/link (or more generally a linkoid) in terms of the states of its constituent linkoids under subdivision, makes evident how its complexity might be influenced by the complexity of the constituent linkoids. The latter combination process is complex and such that the final result is not simply a combination of Jones polynomials of linkoids. However, it is proven in Corollary \ref{prop-bkt-split} that the terms in the formula can be rearranged so that the Jones polynomial of any one of the constituent linkoids appears as a factor. 

\begin{corollary}
Let a link diagram $L$ be subdivided into $2^m$ linkoid pieces $L_i$, where $i \in \{1, 2, \dots , 2^m\}$ and let $\sigma$ be the associated gluing permutation. For any subdivision $L_j$, where $j \in \{1, 2, \dots , 2^m\}$,  the Jones polynomial of $L$ can be expressed as follows :
\begin{equation}
    \begin{split}
       \displaystyle f_L &=
       (-A^3)^{-\sum_{i=1, i \neq j}^{2^m} Wr( L_i)} \sum_{ \mathcal{S}} A^{\sum_{i=1, i \neq j}^{2^m} \alpha(S_{i_{k_i}})} d^{\sum_{i=1, i \neq j}^{2^m} |S_{i_{k_i}}|} f_{(L_j,\sigma \tau_{\mathcal{S}}\sigma)}
    \end{split}
\end{equation}
where $ \mathcal{S}=\left( S_{i_{k_i}} \right)_{\substack{i=1, i\neq j}}^{2^m}$ is a combination of states $S_{i_{k_i}}$ such that each $S_{i_{k_i}}$ represents a state with state permutation  $\tau_{S_{i_{k_i}}}$ in the state sum expansion of $L_i$, \hspace{0.05cm} $\alpha(S_{i_{k_i}})$ is the algebraic sum of the smoothing labels and $|S_{i_{k_i}}|$ is the number of circular components in $S_{i_{k_i}}$. The state permutation of $\mathcal{S}$ is denoted $\tau_{\mathcal{S}} = \prod_{i=1, i\neq j}^{2^m} \tau_{S_{i_{k_i}}}$ and $f_{(L_j,\sigma \tau_{\mathcal{S}}\sigma)}$ is the Jones polynomial of the linkoid $L_j$ with respect to closure permutation $\sigma \tau_{\mathcal{S}}\sigma$.
\label{prop-bkt-split}
\end{corollary}

\begin{proof}
Let $L_j$,  $j \in \{1, 2, \dots , 2^m\}$, be a linkoid in the subdivision of $L={\cup_{\sigma}}_{i=1}^{2^m} L_i$. Consider $L\setminus L_j={\cup_{\sigma}}_{i=1,i\neq j}^{2^m} L_i$ as a linkoid. Its states can be expressed as $ \mathcal{S}=\left( S_{i_{k_i}} \right)_{\substack{i=1, i\neq j}}^{2^m}$ where  $S_{i_{k_i}}$ represents a state with state permutation  $\tau_{S_{i_{k_i}}}$ in the state sum expansion of $L_i$, $i\neq j$. Note that, $ \sigma \tau_{\mathcal{S}}\sigma $ , where $\tau_{\mathcal{S}} = \prod_{i=1, i\neq j}^{2^m} \tau_{S_{i_{k_i}}}$, is a closure permutation on the endpoints of $L_j$. The collection, $(L_j)_{\tau_j}^{\sigma \tau_{\mathcal{S}}\sigma} $  over all $\mathcal{S}$ is a subset of the virtual spectrum of $(L_j)_{\tau_j}$. The Jones polynomial of $L$ can thus be expressed as follows :
\begin{equation} \label{eq-vs-exp}
    \begin{split}
       \displaystyle f_L &=
       (-A^3)^{-\sum_{i=1, i \neq j}^{2^m} Wr( L_i)} \sum_{ \mathcal{S}} A^{\sum_{i=1, i \neq j}^{2^m} \alpha(S_{i_{k_i}})} d^{\sum_{i=1, i \neq j}^{2^m} |S_{i_{k_i}}|} f_{(L_j,\sigma \tau_{\mathcal{S}}\sigma)}
    \end{split}
\end{equation}
\end{proof}

\begin{remark}
    Eq. (4) in Corollary \ref{prop-bkt-split} shows that all the terms in the Jones polynomial of a linkoid $L$ involve as a factor the Jones polynomial of (any) one of its linkoids, $L_j$, with respect to different closure permutations. The latter ones are part of the Jones polynomial of the virtual spectrum of $L_j$ (\cite{Barkataki2024_virtual}). In this context, the overall Jones polynomial $f_L$ can be expressed as a matrix multiplication over its virtual spectrum as: \begin{equation}
     f_L =  \mathbf{M} \mathbf{f_{L_j}},   
    \end{equation}
    where each entry in the $ n \times 1$ vector, $\mathbf{f_{L_j}}$, represents the Jones polynomial of $L_j$ with respect to a specific closure permutation. Thus, $\mathbf{f_{L_j}}$ is the virtual spectrum of the Jones polynomial of the linkoid $L_j$.
    The $1 \times n$ matrix $\mathbf{M}$ captures the weighting factors associated with the different closures of $L_j$. Note that if a closure does not appear in the state sum expansion, the weight is 0. A closure permutation of $L_j$ which appears in the state sum expansion in Equation \ref{eq-vs-exp} may appear in multiple states. Let  $\mathcal{S}^\star$ denote the set of states in which a particular closure occurs, then the  corresponding weight in $\mathbf{M}$ is given by $\displaystyle (-A^3)^{-\sum_{i=1, i \neq j}^{2^m} Wr( L_i)} \sum_{ \mathcal{S^{\star}}} A^{\sum_{i=1, i \neq j}^{2^m} \alpha(S_{i_{k_i}})} d^{\sum_{i=1, i \neq j}^{2^m} |S_{i_{k_i}}|}$. 
\end{remark}


\begin{remark}
In some special cases the Jones polynomial of a knot/link diagram can determined by its subdivided linkoid pieces.

For example, consider a linkoid diagram $L$, such that $L= L_j \cup_\sigma L^{'}$, where $\sigma$ is the corresponding gluing permutation. Let $\tau^{'}$ and $\tau_j$ be the strand permutations of $L^{'}$ and $L_j$ respectively : a)
If $L_j$ is a disjoint union of knotoids, $k_i$, then $f_L = f_{L^{'},\sigma \tau_j \sigma} \times \prod_{i}f_{k_i}; 
 $
b) If the virtual spectrum of $L_j$ only consists of trivial links, then the Jones polynomial of $L$ equals the Jones polynomial of $L^{'}$ with respect to $\sigma$ i.e. $f_L = f_{L^{'},\sigma \tau_j \sigma}$.

It would be interesting to examine if generalization of results, such as those on detecting the unknot in \cite{kauffman2012hard} could be obtained by using Theorem \ref{prop-bkt-split}. 

\label{corr_1piece}
\end{remark}

\section{Conclusion}
\label{sec_discussion}
Creating efficient algorithms for the computation of the Jones polynomial has attracted attention  for both theoretical and practical purposes. In this paper, we introduce the first, to our knowledge, parallel algorithm for the exact computation of the Jones polynomial for any knot(oid) (link(oid)) diagram and collections of open curves in 3-space.  The algorithm reduces the computational time complexity of the Jones polynomial by an exponential factor depending on the number of processors. This result is enabled by a combinatorial study of linkoids, which results in closed formulas for the Jones polynomial of a knot(oid)/link(oid) in terms of its subdivided pieces. The method is general and could be applied to develop parallel algorithms for more topological invariants, such as the Arrow polynomial and Khovanov homology.

A \texttt{GitHub} repository, \textit{Parallel Jones Polynomial}, \url{https://github.com/Parallel-Jones-Polynomial/ParallelJones}, employing our proposed algorithm is published, along with an explanatory usage description. We demonstrated the practical implications of this algorithm by computing the Jones polynomial in serial and in parallel for knots and for linear polymer chains in a melt obtained from Molecular Dynamics simulations. 

The general framework we introduce here for developing a parallel algorithm for the Jones polynomial is also relevant for other questions in knot theory. For example, a question in knot theory, which could also have practical implications, is whether the Jones polynomial of the knot is related to the Jones polynomial of the constituent linkoids.  We prove that the Jones polynomial of a linkoid can in fact be expressed as a linear combination of the Jones polynomial of the virtual spectrum  of any one of its pieces. 


\section{Acknowledgement}
Kasturi Barkataki and Eleni Panagiotou acknowledge support by the Natoinal Science Foundation, NSF DMS-1913180, NSF CAREER 2047587 and National Institutes of Health, grant number R01GM152735-01.

\section{Declarations}
\textbf{Conflict of Interest} The authors have no competing interests to declare that are relevant to the content of
this article.

\noindent \textbf{Ethical Approval} Not applicable.

\section{Appendix}

In Section \ref{sec-state-perms-seg-cyc}, we explore the connection between the number of strands and the number of distinct state permutations in linkoid diagrams and also discuss how a state permutation exactly defines the number of segment cycles in a given linkoid state. In Section \ref{sec-treewidth}, we explore a seamless integration of the methods of the parallel algorithm introduced in this paper with the fixed parameter tractable algorithm discussed in \cite{burton-homflypt-fpt}.

\subsection{Distinct state permutations and segment cycles in linkoids}
\label{sec-state-perms-seg-cyc}

A labelled linkoid diagram can be associated with a \textit{strand permutation} and a set of \textit{state permutations} on its endpoints, which are defined as follows :

\begin{definition}  (\textit{Strand Permutation and State Permutation of a Linkoid} \cite{Barkataki2022,Barkataki2024_virtual}) Given a labelled linkoid diagram, its \textit{strand permutation} is defined to be a
permutation on its labelled endpoints, pairing each
endpoint with its corresponding partner on the same
component of the linkoid. A choice of smoothing at all the crossings of a linkoid diagram determines a state in its bracket polynomial state sum expansion. Any state of a linkoid consists solely of open arcs with no crossings and possibly a collection of circles. The arcs of a state determine permutations on the endpoints of the linkoid, which is called a \textit{state permutation}.  
\end{definition}

Let $L$ be a linkoid diagram with $c$ components/strands and $n$ crossings.
There are two ways of choosing a smoothing at any crossing of $L$, i.e., \raisebox{-5pt}{\includegraphics[width=.025\linewidth]{fig/cross30.png}} or \raisebox{-5pt}{\includegraphics[width=.025\linewidth]{fig/cross20.png}}, which means that $L$ has $2^n$ states in total. However, it is possible for multiple states to correspond to the same state permutation. Estimates such as the total number of states, $2^n$, and the total number of $c$-transpositions on the endpoints, $\frac{(2c)!}{2^c c!}$, serve as loose upper bounds and significantly overestimate the number of distinct state permutations in a linkoid diagram. In the following section, we provide a tight upper bound on the number of distinct state permutations in a linkoid in terms of the number of its components (strands).

\subsubsection{Distinct state permutations of a linkoid diagram}

 In Proposition \ref{thm-const-CkLk}, it is shown that there exists a linkoid with $c$ components where $c \in \mathbb{N}$ such that it has $\frac{1}{c+1} \binom{2c}{c}$ distinct state permutations. This immediately provides a tight upper bound on the number of distinct state permutation of any $c$-tangle (or linkoid with $c$ components), as shown in Corollary \ref{corr-any-ktang}.
\begin{proposition}
Let $L_c$ be a link diagram with $c$ components, formed by arranging $c$ straight lines (no two are parallel) in the plane such that each pair of lines intersects exactly once, resulting in $\binom{c}{2}$ crossings. Upon resolving all crossings using $A$ or $B$-smoothings, the diagram $L_c$ gives rise to $\mathbf{C}_c$ distinct state permutations, where $\mathbf{C}_c = \frac{1}{c+1} \binom{2c}{c}$ is the $c^{th}$ Catalan number.
\label{thm-const-CkLk}
\end{proposition}

\begin{proof}
The link diagram $L_c$ is a $c$-tangle with $2c$ endpoints fixed in a specific order around the boundary of a disk. Without loss of generality, choose a point $p$ on the boundary. For any smoothing of the crossings in $L_c$, $p$ can be paired with another point at an odd distance, i.e. $2i+1$, along the boundary. Pairing $p$ in this way divides the disk into two regions with $2i$ and $2(c-1)-2i$ endpoints, respectively for $0\leq i \leq c-1$. Let $\mathcal{P}_{2c}$ denote the number of distinct state permutations in the state sum expansion of $L_c$. This yields the recursion:
$$
\mathcal{P}_{2c} = \sum_{i=0}^{c-1} \mathcal{P}_{2i} \mathcal{P}_{2c-2i-2}.
$$
Setting $\mathcal{P}_{2i} = \mathbf{C}_i$, the recursion matches the Catalan number recurrence (\cite{conway1998book}). Hence, $\mathcal{P}_{2c} = \frac{1}{c+1} \binom{2c}{c}$.
\end{proof}

A link-type linkoid, where all endpoints lie in the same region of the diagram, corresponds to a classical $c$-tangle. In contrast, a pure linkoid (where not all endpoints lie in the same region of the diagram) corresponds to a virtual $c$-tangle, with virtual crossings arising when the endpoints are rearranged to lie within the same region of the diagram.

\begin{corollary}
    Let $L$ be an arbitrary $c$-tangle (classical or virtual). Then $L$ has at most $\mathbf{C}_c$ distinct state permutations.\label{corr-any-ktang}
\end{corollary}

\begin{proof}
\textbf{Case 1:} Let $L$ be a classical $c$-tangle (consisting only of classical crossings). Since the endpoints of $L$ can be arranged in a specific order around the disk, the possible pairings of endpoints can be analyzed as in Proposition \ref{thm-const-CkLk}. Thus, any classical $c$-tangle has at most $\mathbf{C}_c$ distinct state permutations.

\noindent \textbf{Case 2:} Let $L$ be a virtual $c$-tangle (consisting of both classical and virtual crossings), and let $L'$ be a classical $c$-tangle obtained by replacing all virtual crossings of $L$ with classical crossings. Since $L'$ has more classical crossings than $L$, the number of states in the state sum expansion of $L'$ is greater than that of $L$. Therefore, the distinct state permutations of $L^{'}$ is at least as many as those of $L$ which implies that $L$ has at most $\mathbf{C}_c$ distinct state permutations.
\end{proof}

If a given linkoid diagram has no loop bounded by its arcs, then all its states correspond to distinct state permutations. In fact, the number of strands is one more than the number of crossings in the case of a linkoid diagram with no loops and no disjoint component as shown in the following lemma.

\begin{lemma}
The number of strands in a linkoid diagram with $x$ crossings and no disjoint components, in which arcs connecting crossings do not form loops, is $x+1$. \label{lem-max-c-for-n}
\end{lemma}
\begin{proof}
 Each crossing of a linkoid diagram is associated with 2 strands and 4 endpoints when it is considered to be disjoint from all other crossings. When the crossings are viewed as the nodes of a graph, a path of length $x-1$ gives rise to a connected diagram (no disjoint component) without loops/cycles. Every edge of this path merges two distinct endpoints resulting in a diagram with $4x - 2(x-1) = 2x+2$ endpoints. Therefore, the resultant diagram consists of $x+1$ strands in total.
\end{proof}

The following remark provides examples emphasizing the relationship between the number of crossings, strands, and distinct state permutations in linkoid diagrams, showing that the number of distinct state permutations can be controlled by minimizing the number of strands.

\begin{remark}
(i) The number of crossings in the linkoid $L_c$ analyzed in Proposition \ref{thm-const-CkLk} is $\binom{c}{2}$, and the number of states in the bracket expansion is $2^{\frac{c(c-1)}{2}}$. In contrast, the number of distinct state permutations is given by $\frac{1}{c+1} \binom{2c}{c}$, which is strictly less than $2^{\frac{c(c-1)}{2}}$ for all $c > 2$.
(ii) Consider a linkoid diagram $L$ with $n$ crossings, subdivided into $2^m$ pieces, each containing either $\lfloor n / 2^m \rfloor$ or $\lfloor n / 2^m \rfloor + 1$ crossings. Let $c$ denote the number of strands in a subdivided piece such that the number of distinct state permutations is  $\mathbf{C}_c$. Then, setting $\binom{c}{2} = n / 2^m$ implies that
$
c = 1/2 + \sqrt{1/4 + n/2^{m-1}}
$.
Now, consider another linkoid diagram $\tilde{L}$ with $\tilde{c}$ strands, where $\tilde{c}-1$ strands are  parallel, and the $\tilde{c}^{th}$ strands intersects each of the $\tilde{c}-1$ strands exactly once. Then, setting the number of crossings in $\tilde{L}$ to be equal to $n / 2^m$ implies that $\tilde{c}=n/2^m + 1$. While both $L$  and $\tilde{L}$ have the same number of crossings, the latter has more strands than the former, resulting in a larger set of distinct state permutations. 
\end{remark}

\subsubsection{Number of segment cycles in a linkoid state}

The strand permutation and the state permutation of a linkoid diagram is used to determine the number of segment cycles, a quantity needed to define the bracket polynomial of linkoids,  in a given state of the linkoid diagram (\cite{Barkataki2022,Barkataki2024_virtual}). In the following, we discuss the total number of segment cycles, given a state and a closure permutation of a linkoid diagram.

\begin{proposition}
Let $E$ be the set of labelled endpoints of a linkoid diagram with $2n$ endpoints, $\sigma$ be a closure permutation and $\tau_S$ be the state permutation corresponding to a state, $S$, in the state sum expansion of the linkoid diagram. The number of distinct segment cycles in $S$, denoted by $\displaystyle |E/\langle \tau_S, \sigma \rangle|$, has the following expression :
\begin{equation}\begin{split}
   \displaystyle |E/\langle \tau_S, \sigma \rangle| &= \frac{1}{2 |\sigma \tau_S|}\left[ 2n + \sum_{p=1}^{|\sigma \tau_S|-1}  \sum_{\{\mathcal{C}_i : p \equiv 0 \operatorname{mod} l(\mathcal{C}_i)\} } l(\mathcal{C}_i) \right ].\\
   \end{split}\label{eq-C-count2}
\end{equation}
where $|\sigma \tau_S|$ is the order of $\sigma \tau_S$,  $\mathcal{C}_i$ represent cycles in the cycle decomposition, $\prod_i \mathcal{C}_i$, of $\sigma \tau_S$, and $l(\mathcal{C}_i)$ denotes the length of $\mathcal{C}_i$.\label{count-prop}
\end{proposition}

\begin{proof}
    Since, $\tau_S, \sigma$ are transpositions, $\langle \tau_S, \sigma \rangle$ is isomorphic to the dihedral group, $D_{2|\sigma \tau_S|}$. A presentation of this group is : $\displaystyle \langle \sigma\tau_S, \sigma \mid (\sigma \tau_S)^{|\sigma \tau_S|} = \sigma^2 = e \rangle$ and a general element $g$ in this group can be expressed as $g = (\sigma \tau_S)^p \sigma^q$, where $0 \leq p < |\sigma \tau_S|$ and $\quad 0\leq q < 2$. The number of orbits (distinct segment cycles) due to the action of $\langle \tau_S, \sigma \rangle$ on $E$  is given as :
\begin{equation}\begin{split}
   \displaystyle |E/\langle \tau_S, \sigma \rangle| &=  \frac{1}{2|\sigma \tau_S|}\sum_{p=0}^{|\sigma \tau_S|-1}\sum_{q =0}^{1} \left|E^{(\sigma \tau_S)^p \sigma^q}\right|\\&=\frac{1}{2|\sigma \tau_S|}\left[ 2n + \sum_{p=1}^{|\sigma \tau_S|-1}\left|E^{(\sigma \tau_S)^p}\right| +  \sum_{p =1}^{|\sigma \tau_S|-1} \left|E^{(\sigma \tau_S)^p \sigma}\right|\right],\\
   \end{split}\label{eq-C-count}
\end{equation}
where, $\left|E^{(\sigma \tau_S)^p \sigma^q}\right|$ denotes the set of elements in $E$ fixed by $(\sigma \tau_S)^p \sigma^q$. If $|\sigma \tau_S|$ is odd, all reflections (elements of order 2) in $\langle \tau_S, \sigma \rangle$ are conjugates of each other. Since $\sigma$ is a reflection and has no fixed points, hence no other reflection has fixed points. If $|\sigma \tau_S|$ is even, a reflection in $\langle \tau_S, \sigma \rangle$ is either of the form $ (\sigma \tau_S)^{2i} \sigma $ which is conjugate to $\sigma$ or of the form  $(\sigma \tau_S)^{2i+1} \sigma $ which is conjugate to $\tau_S$ (where $i \in \mathbb{Z}$). Since, neither $\sigma$ nor $\tau_S$ has fixed points, hence $\displaystyle  \left|E^{(\sigma \tau_S)^p \sigma}\right| = 0$ for all $p$. Therefore, Equation \ref{eq-C-count} simplifies to :
\begin{equation}\begin{split}
   \displaystyle |E/\langle \tau_S, \sigma \rangle| &= \frac{1}{2 |\sigma \tau_S|}\left[ 2n + \sum_{p=1}^{|\sigma \tau_S|-1}\left|E^{(\sigma \tau_S)^p}\right| \right ].\\
   \end{split}\label{eq-C-count2}
\end{equation}
Let $\displaystyle \prod_i \mathcal{C}_i$ denote the cycle decomposition of $\sigma \tau_S$ and $l(\mathcal{C}_i)$ denote the length of the cycle $\mathcal{C}_i$. If $p$ is a multiple of $l(\mathcal{C}_i)$, for some cycle $\mathcal{C}_i$, then all the elements of $\mathcal{C}_i$ are fixed by $(\sigma \tau_S)^p$. Thus, the required expression is obtained by substituting $\displaystyle \left|E^{(\sigma \tau_S)^p}\right| = \sum_{\{\mathcal{C}_i : p \equiv 0 \operatorname{mod} l(\mathcal{C}_i)\} } l(\mathcal{C}_i)$.
\end{proof}

\subsection{The parallel algorithm and treewidth parameter of link(oid) diagrams}
\label{sec-treewidth}

An alternative method of computing the Jones polynomial is by applying the following skein relation to an oriented knot/link diagram (\cite{Jones1985}):
$$t^{-1}f_{L_+}-t f_{L_-}=\left(t^{1/2}-t^{-1/2}\right)f_{L_0} \hspace{0.05cm}.$$
where the diagrams $L_+,\hspace{0.05cm}L_-$ and $L_0$ are identical everywhere except at one crossing as shown below:
\begin{center}
    \includegraphics[scale=0.12]{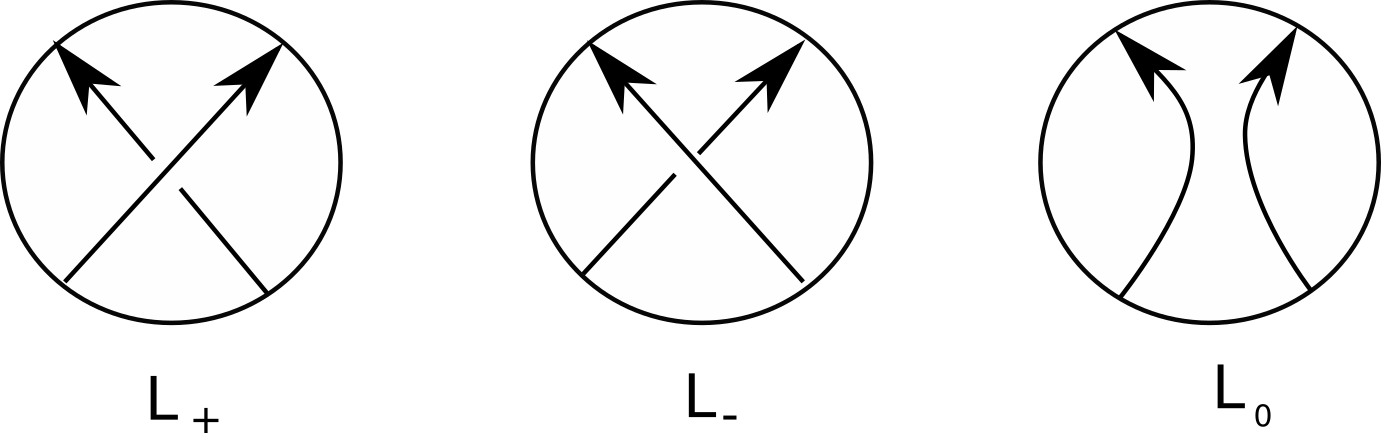}
\end{center}
This skein template on a knot/link diagram allows to build a decision tree in which the leaves are obtained from the original link diagram by switching and/or splicing some crossings to obtain descending diagrams. The decision tree is constructed by systematically traversing the link's arcs in numerical order (ordering is assigned arbitrarily), modifying under-crossings by either splicing or switching to generate different link configurations. The traversal continues until all arcs are visited, resulting in descending diagrams at the leaves of the tree.

Let $L$ denote a link diagram and $\Gamma(L)$ the is the
directed planar 4-valent multigraph whose vertices are the crossings of $L$.  In \cite{burton-homflypt-fpt}, a Fixed Parameter Tractable (FPT) algorithm is described that processes the link diagram $L$ using a nice tree decomposition $T$ of its graph $\Gamma(L)$. A tree decomposition of a graph represents its structure as a tree whose nodes correspond to subsets of vertices, while a nice tree decomposition is a special version where the tree is binary and consists of structured nodes (leaf, introduce, forget, and join). Each crossing of $L$ is associated with a unique forget bag of $T$ and by ordering arcs of $L$ based on the position of the corresponding forget bags, the algorithm avoids explicitly encoding starting arcs of the link traversal, ensuring configurations to depend only on treewidth.  The treewidth of every
planar graph on $n$ vertices (in particular the graph $\Gamma(L)$ associated with a knot/link $L$ diagram with $n$ crossings) is at worst $\mathcal{O}(\sqrt{n})$ (\cite{lipton1979separator}). It is proven in  \cite{burton-homflypt-fpt} that given a link diagram $L$ with $n$ crossings, it is possible to compute the HOMFLY-PT polynomial (and hence the Jones Polynomial) of $L$ by using the FPT algorithm in sub-exponential time $e^{\mathcal{O}(\sqrt{n}\log n)}$.

In order to define a nice tree decomposition of linkoid diagrams, we need to determine what a descending linkoid is:

\begin{definition}(\textit{Descending Linkoid Diagram})
Let $L$ be a linkoid diagram equipped with a fixed ordering of its arcs. $L$ is said to be a \textit{descending linkoid diagram} if, when traversing the arcs of $L$ in the given order, each crossing is first encountered at the over-crossing strand. 
\label{def-desc-linkoid}
\end{definition}

\begin{remark}
    Notice that a descending linkoid may not be a trivial linkoid.
\end{remark}

Nice tree decompositions of linkoid diagrams can be constructed in a way similar to the construction of nice tree decompositions of links. In particular, for a linkoid $ L={\cup_{\sigma}}_{i=1}^{2^m} L_i$, which is a gluing of linkoids $L_i$, where $i=1, \dotsc, 2^m$ of $\frac{n}{2^m}$ crossings each, nice tree decompositions can be constructed for the graph, $\Gamma(L_i)$, corresponding to each of the individual $L_i$, assuming the ordering of arcs of $L_i$ is that induced by the ordering of the arcs of $L$. Since, $\Gamma(L_i)$ is a graph on $\frac{n}{2^m}$ vertices, its treewidth is at worst $\mathcal{O}(\sqrt{\frac{n}{2^m}})$.

By Definition \ref{def-desc-linkoid}, each leaf (which is a  linkoid diagram) in the decision tree of $L_i$ is a descending linkoid diagram. 

\bibliographystyle{plain}
\bibliography{paperDatabase}

\end{document}